\documentclass[11pt, reqno]{amsart}
\usepackage{amssymb}
\usepackage{amsmath}
\usepackage{amsthm}
\usepackage[utf8]{inputenc}
\usepackage{longtable}
\usepackage{cite}
\usepackage{amsfonts,amssymb,amsthm,amsmath,eucal,url}
\usepackage{pgf}
 \usepackage{array}
 \usepackage{pstricks}
 \usepackage{pstricks-add}
 \usepackage{pgf,tikz}
 \usetikzlibrary{automata}
 \usetikzlibrary{arrows}
 \usepackage{indentfirst}
 \usepackage{float}
 \definecolor{uuuuuu}{rgb}{0.26666666666666666,0.26666666666666666,0.26666666666666666}
\definecolor{qqqqff}{rgb}{0.,0.,1.}

\theoremstyle{plain}
\newtheorem{thm}{Theorem}[section]
\newtheorem{theorem}[thm]{Theorem}

\newtheorem{corollary}[thm]{Corollary}
\newtheorem{conjecture}[thm]{Conjecture}

\theoremstyle{definition}
\newtheorem{definition}[thm]{Definition}

\newtheorem{example}[thm]{Example}

\newtheorem{thevarthm}[thm]{\varthmname}

\newenvironment{varthm*}[1]{\trivlist\item[]{\bf #1.}\it}{\endtrivlist}

\renewcommand\geq{\geqslant}

\renewcommand\leq{\leqslant}

\let\tilde=\widetilde

\newcommand\be{\begin{eqnarray*}}
\newcommand\ee{\end{eqnarray*}}

\newcommand\C{\mathbb C}

\newcommand\newop[2]{\def#1{\mathop{\rm #2}\nolimits}}
\newop\log{log}
\newop\ord{ord}
\newop\Gal{Gal}
\newop\SL{SL}
\newop\Bl{Bl}
\newop\mult{mult}
\newop\mass{mass}
\newop\div{div}
\newop\codim{codim}
\newop\sing{sing}
\newop\vdim{vdim}
\newop\edim{edim}
\newop\Ass{Ass}
\newop\size{size}
\newop\reg{reg}
\newop\satdeg{satdeg}
\newop\supp{supp}
\newop\Neg{Neg}
\newop\Nef{Nef}
\newop\Nefh{Nef_H}
\newop\Eff{Eff}
\newop\Zar{Zar}
\newop\MB{MB}
\newop\MBxC{MB\mathit{(x,C)}}
\newop\NnB{NnB}
\newop\Bigg{Big}
\newop\Sing{Sing}
\newop\Effbar{\overline{\Eff}}

\newcommand\wtilde[1]{\widetilde{#1}}

\begin{document}
\title[On the local negativity of surfaces]{On the local negativity of surfaces with numerically trivial canonical class}
\author{Roberto Laface \and Piotr Pokora}
\date{\today}
   
\address{Roberto Laface \newline Institut f\"ur Algebraische Geometrie,
    Leibniz Universit\"at Hannover,
    Welfengarten 1, 
    D-30167 Hannover, Germany
}
	\email{laface@math.uni-hannover.de}

\address{Piotr Pokora \newline Instytut Matematyki,
   Pedagogical University of Cracow,
   Podchor\c a\.zych 2,
   PL-30-084 Krak\'ow, Poland.}
   \curraddr{Institut f\"ur Mathematik,
   Johannes Gutenberg Universit\"at Mainz,
   Staudingerweg 9,
   D-55099 Mainz, Germany}
   \email{piotrpkr@gmail.com}

\maketitle

\thispagestyle{empty}
\begin{abstract}
In this note we study the local negativity for certain configurations of smooth rational curves in smooth surfaces with numerically trivial canonical class. We show that for such rational curves there is a bound for the so-called local Harbourne constants, which measure the local negativity phenomenon. Moreover, we provide explicit examples of interesting configurations of rational curves in some K3 and Enriques surfaces and compute their local Harbourne constants.
\end{abstract}


\section{Introduction}

In this note we continue studies on the local negativity of algebraic surfaces. In last years there is a resurgence around questions related to negative curves on algebraic surfaces. One of the most challenging problems is the Bounded Negativity Conjecture (BNC in short).

\begin{conjecture}[Bounded Negativity Conjecture]
Let $X$ be a smooth projective surface defined over a field of characteristic zero. Then there exists an integer $b(X) \in \mathbb{Z}$ such that for all \emph{reduced} curves $C \subset X$ one has $C^{2} \geq -b(X)$.
\end{conjecture}

It is easy to see that the number $b(X)$ can be arbitrary large depending on $X$. In order to observe this phenomenon, consider the blow-up $X_s$ of the projective plane $\mathbb{P}^{2}_{\mathbb{C}}$ along $s\gg 0$ mutually distinct points $P_1, \dots, P_s$ lying on a line $l$. It is easy to see that the strict transform of $l$ has the form $\tilde{l}_{s} = H - E_1 - \cdots - E_s$, where $H$ is the pull back of $\mathcal{O}_{\mathbb{P}^2_{\mathbb{C}}}(1)$ and $E_1, \dots, E_s$ are the exceptional divisor, and $\tilde{l}_{s}^{2} = -s+1$. Moreover, it is not difficult to see that $b(X_s) = s(s-1)$. In order to avoid such situations we can define an asymptotic version of self-intersection numbers, in the case of arbitrary blow-ups of the projective plane one can divide the self-intersection number of a reduced curve by the number of point we blown up our surface. It turns out that this approach is more effective in the context of the BNC. 

It is worth pointing out that the BNC is widely open, but there are some cases for which $b(X)$ is known. It can be shown that the BNC is true for minimal models with Kodaira dimension equal to zero. In particular, we know that for surfaces with numerically trivial canonical class every reduced and irreducible curve $C$ has $C^{2} \geq -2$, which is a consequence of the adjuction formula. However, it is not known whether the BNC still holds for blow-ups of those surfaces along sets of points.

The main aim of this note is to study the BNC for blow-ups of surfaces with numerically trivial canonical class from a point of view of Harbourne constants, which were introduced in \cite{BdRHHLPSz}, and allow to measure the negativity phenomenon asymptotically. Before we define the main object of this paper, we need to recall some standard notions.

\begin{definition}[The numbers $t_i$]
   Let ${C}$ be a configuration of finitely many mutually distinct smooth curves in a projective surface $X$.
   We say that $P$ is an $r$-fold point of the configuration ${C}$,
   if it is contained in exactly $r$ irreducible components of ${C}$. The union of all $r$-fold points
   $P\in {C}$ for $r\geq 2$, is the set $\Sing({C})$ of singular points of ${C}$.
   We set the number $t_r=t_r({C})$ to be the number of $r$-fold points in ${C}$.
\end{definition}

We will mostly deal with configurations of smooth curves having only transversal intersection points. Letting ${C} = C_1 + \dots + C_n$ be a configuration of such curves on $X$, consider the blow-up of $X$ along $\Sing({C})$, with the exceptional divisors $E_1,\ldots,E_s$, where $s$ is the cardinality of ${\rm Sing}({C})$. Let $\wtilde{C}$ be the strict transform of $C$. Then the divisor
   $$\wtilde{C}+E_1+\cdots+E_s$$
is a \textit{simple normal crossing} divisor on $Y$, meaning that
\begin{enumerate}
\item it is reduced;
\item its irreducible components are all smooth;
\item there are at most two irreducible components going through a point of the divisor (i.e.,~at the singular points, the divisor locally looks like the intersection of the coordinate axes in $\C^2$).    
\end{enumerate}

Simple normal crossing divisors come pretty handy, as it is quite easy to compute their Chern numbers, a fact that we will employ in our computations (see the proof of Theorem \ref{k3}). In the present note, we are interested in (local) Harbourne constants attached to transversal configrations, i.e.,~configurations of curves such that all irreducible components are smooth and all intersections are transversal. They can be viewed as a way to measure the average negativity coming from singular points in the configuration. 
\begin{definition}[Local Harbourne constants of a transversal configuration]\label{def:H-constant TA}
   Let $X$ be a smooth projective surface.
   Let $C=\sum_{i=1}^\tau C_i$ be a transversal configuration of curves in $X$ with $s=s(C)$ singular points. Let $\widetilde{C}=\sum_{i=1}^\tau \widetilde{C_i}$ be the strict transform of $C$ in the blow-up $Y$ of $X$ at $s$ singular points of $C$, $\widetilde{C_i}$ being the strict transform of $C_i$ in $Y$.
   The rational number
   \begin{equation}\label{eq:TA Harbourne constant}
      h(X;C)=h(C)=\frac{\widetilde{C}^2}{s} = \frac1s\left(C^2-\sum\limits_{P\in \Sing(C)}m_P^2\right),
   \end{equation}
where $m_P$ is the multiplicity of the divisor $C$ at $P$, is the \emph{local Harbourne constant} of $C\subset X$.
\end{definition}

We can also express the local Harbourne constant in terms of the $t_i$'s as follows:
\[ h(C) = \frac{1}{s}\left(C^2-\sum\limits_{r \geq 2}r^2t_r\right).\]
Notice that the notion of local Harbourne constant is a natural variation of the original Harbourne constant in \cite{BdRHHLPSz}, the latter to be understood as a global version of the first one. In fact, the local Harbourne constant defined in the present paper considers \emph{only one reduced curve}, while the Harbourne constant of \cite{BdRHHLPSz} is computed by considering all curves in $X$ at once.

\section{A bound on local Harbourne constants}
We would like to focus on the case of configurations of smooth rational curves in surfaces with the numerically trivial canonical class having only transversal intersection points. We start with the following result which is a generalization of a result due to Miyaoka \cite[Section~2.4]{Miyaoka1}.
\begin{theorem}
\label{k3}
Let $X$ be a smooth complex projective surface with numerically trivial canonical class and let ${C} \subset X$ be a configuration of smooth rational curves having $n$ irreducible components and only transversal intersection points. Then
$$4n - t_{2} + \sum_{r \geq 3} (r-4)t_{r} \leq 3c_2(X) \leq 72.$$
\end{theorem}
\begin{proof}
Let ${C}= C_1 + \dots + C_n$ be a configuration of smooth rational curves in $X$. If $\text{Sing}({C})$ denotes the set of singular points of the configuration, we define $S=\lbrace p_j \rbrace_{j=1}^k$ to be the subset of points in $\text{Sing}({C})$ with multiplicity $\geq 3$. Consider the blow-up of $X$ at the points of $S$, namely
\[ \sigma : Y \longrightarrow X;\]
under pull-back along $\sigma$, the configuration ${C}$ on $X$ yields a configuration $\sigma^*{C}$ which consists of the strict transforms of the $C_i$'s and the exceptional divisors. Notice that $\sigma^*{C}$ is again a configuration of smooth rational curves that admits double points only as singularities. Following \cite[Section~2.4]{Miyaoka1}, we set $M:= \tilde{C}_1 + \cdots + \tilde{C}_n$. The idea is to use the Bogomolov-Miyaoka-Yau inequality
\[3c_2(Y) - 3e(M) \geq (K_Y + M)^2,\]
and thus we now need to compute the terms in the above inequality. We see that
\begin{align*}
& c_2(Y) = c_2(X) +k,\\
& e(M) = 2n - t_2,
\end{align*}
which yield $c_2(Y) - e(M) = c_2(X)+k-2n+t_2$.  If $E:=\sum_{j=1}^k E_j$ is the sum of all exceptional divisors, we have
\begin{align*}
& K_Y+M = \sigma^*(K_X +C) - \sum_{j=1}^k (m_j-1)E_j,\\
& (K_Y+M)^2 = C^2 - \sum_{j=1}^k (m_j-1)^2.
\end{align*}
Notice that
\[ K_Y + M = (\sigma^*K_X + E) + M = \sigma^*K_X + (E+M),\]
and as $K_X$ is numerically trivial, also $\sigma^*K_X$ is, and thus $K_Y + M$ is numerically equivalent to an effective divisor. This allows us to use the Bogomolov-Miyaoka-Yau inequality according to \cite[Corollary 1.2]{Miyaoka1}.
Moreover,
\begin{align*}
C^2 &= -2n + 2 \sum_{i<j} C_i . C_j \\
&= -2n + 2 \sum_{r \geq 2} {r\choose 2} t_r \\
&= -2n + 2t_2 +2\sum_{r \geq 3} {r\choose 2} t_r.
\end{align*}
It follows that
\begin{align*}
(K_Y+M)^2 &= -2n + 2 t_2 + 2 \sum_{r \geq 3} {r\choose 2} t_r - \sum_{j=1}^k (m_j -1)^2\\
&= -2n + 2t_2 + 2 \sum_{j=1}^k {m_j \choose 2} - \sum_{j=1}^k (m_j-1)^2\\
&= -2n + 2 t_2 + \sum_{j=1}^k (m_j-1).
\end{align*}
By plugging into the Bogomolov-Miyaoka-Yau inequality, we see that
\begin{align*}
3c_2(X) &\geq 4n - t_2 -3k + \sum_{j=1}^k (m_j-1) \\
&= 4n - t_2 + \sum_{j=1}^k (m_j -4) \\
&= 4n - t_2 + \sum_{r \geq 3} (r -4)t_r,
\end{align*}
and the result follows from the fact that $c_2(X) = 12 (1-q(X) +p_g(X)) \leq 24$, and equality holding if and only if $X$ is a K3 surface.
\end{proof}

Now we can prove the main result of this paper.
\begin{theorem}\label{bound}
Let $X$ be a smooth complex projective surface with numerically trivial canonical class and let ${C}$ be a configuration of smooth rational curves having $n$ irreducible components and only $s \geq 1$ transversal intersection points. Then
$$h(X; {C}) \geq -4 + \frac{2n + t_{2} -3c_2(X)}{s} \geq -4 + \frac{2n + t_{2} -72}{s}$$
\end{theorem}
\begin{proof}
If ${C} = C_1 + \dots + C_n$, the Harbourne constant is computed by $\tilde{C}^2/s$, where $\tilde{C}$ is its strict transform in the blow-up at $s$ singular points of the configuration. We observe that
\[\tilde{C}^2/s = \frac{C^2 - \sum_j m_j^2}{s} = \frac{-2n + I_d - \sum_{r \geq 2} r^2t_r}{s},\]
where $I_d := 2 \sum_{i<j} C_i . C_j$ is the number of incidences of the configuration ${C}$. By the definition of $I_d$, it is straightforward to see that
\[I_d - \sum_{r \geq 2} r^2t_r = - \sum_{r \geq 2} rt_r,\]
and moreover we can rephrase the bound in Theorem \ref{k3} in the following way:
\[-\sum_{r \geq 2} r t_r \geq -4s + 4n + t_2 - 3c_2(X).\]
This yields
\[  h(X; {C})   \geq -4 + \frac{2n+t_2 -3c_2(X)}{s},\]
which complete the proof.
\end{proof}

Let us now define the following number.
\begin{definition}
Let $X$ be a smooth complex projective surface with numerically trivial canonical class. The real number
  $$H_{\text{rational}}(X)=\inf_C h(X;C),$$
  where the infimum is taken over all transversal configurations of smooth rational curves $C\subset X$
   is the \emph{global rational Harbourne constant} of $X$.
\end{definition}
\begin{corollary}
In the setting of Theorem \ref{bound}, one has
$$H_{\text{rational}}(X) \geq -45.$$
\end{corollary}

\begin{proof}
	Let us consider a configuration $C= C_1 + \cdots + C_n$ of $n$ rational curves in $X$. If $s \geq 2$, then we obtain the following chain of inequalities:
    \[ h(X,C) \geq -4 +\frac{2n + t_{2} -72}{s} = -4 - \dfrac{72}{s} + \dfrac{2n+t_2}{s} \geq -40. \]
    
    We are now left with dealing with the case $s=1$. We assume firstly that our configuration $C$ is connected (i.e.~there is no isolated rational curve that does not intersect others). In this situation, Theorem \ref{bound} has the form $h(X,C) \geq -76 + 2n +t_2$. Using the local Harbourne constant of $C$, we get
    \[ -2n - \sum_{r \geq 2} rt_r \geq -76 + 2n + t_2. \]
    As $\sum_{r\geq 2} rt_r =n$, we get that $n \leq 15$, which implies $h(X,C)= -3n \geq -45$. 
    
    Suppose now that we are given an arbitrary configuration, and let $m$ be the positive integer such that there exist exactly $n-m$ curves of $C$ which do not intersect any other component of $C$ (in other words, the configuration contains $n-m$ isolated rational curves). By looking at the local Harbourne constant of $C$, it is straightforward to see that
    \[ h(X,C) = -2n - \sum_{r\geq 2} r t_r =  -2m - \sum_{r \geq 2} r t_r -2(n-m) = -m -2n. \]
    
    Now, by the argument for connected configurations, we get $m \leq 15$. By Theorem \ref{bound}, we deduce the following:
    \[ -m-2n = h(X,C) \geq -76 +2n+t_2 \geq -76 + 2n, \]
   which gives $4n+m \leq 76$. This implies, in particular, that $n \leq 18$. Therefore, we obtain the following restrictions on the configuration $C$:
    \[ n \leq 18, \qquad m \leq 15, \qquad 4n+m \leq 76. \]
    
    If $n=18$, then $2 \leq m \leq 4$. In fact, $3 \leq m \leq 4$, as otherwise there would be $n-m =16$ mutually disjoint rational curves on $X$, which are also disjoint from the subconfiguration of $C$ consisting of $m$ lines meeting at a single point, contradicting a result due to Miyaoka \cite[Proposition~2.1.1]{Miyaoka1}. In this situation, one has $h(X,C) \geq -40$.
    
    If $n=17$, then $2 \leq m \leq 8$, from which one sees that $h(X,C) \geq -42$. If $n = 16$, we get that $2 \leq m \leq 12$, which implies $h(X,C) \geq -44$ . Finally, if $n \leq 15$, then necessarily $2 \leq m \leq 15$ (by the argument for connected configurations), and thus $h(X,C) \geq -45$.
\end{proof}

\section{Examples}

We now give some examples of interesting configurations of smooth rational curves on K3 and Enriques surfaces. We will use Theorem \ref{k3} to give a lower bound for their Harbourne constants. 

\begin{example}[Six general lines in $\mathbb{P}^2$]
In the complex projective plane $\mathbb{P}^2$, consider six lines in general position, and denote this configuration by $\mathcal{L}$. 
\begin{center}
\begin{figure}[h]
\begin{tikzpicture}[line cap=round,line join=round,>=triangle 45,x=10,y=10, scale=0.8]
\clip(-4.28,-4.28) rectangle (22.82,6.84);
\draw [domain=-4.28:22.82] plot(\x,{(-0.6648--0.02*\x)/6.9});
\draw [domain=-4.28:22.82] plot(\x,{(-24.1472--4.24*\x)/2.92});
\draw [domain=-4.28:22.82] plot(\x,{(--52.68-4.22*\x)/3.98});
\draw [domain=-4.28:22.82] plot(\x,{(--6.587264134273421-2.7209310968454425*\x)/12.312792177466502});
\draw [domain=-4.28:22.82] plot(\x,{(--42.80957838345586-2.5087039120105596*\x)/-12.8853392301207});
\draw [domain=-4.28:22.82] plot(\x,{(--40.422439238249204-5.692700839257268*\x)/-0.0710546587099703});
\end{tikzpicture}
\caption{Six lines in general position in $\mathbb{P}^2$.}
\end{figure}
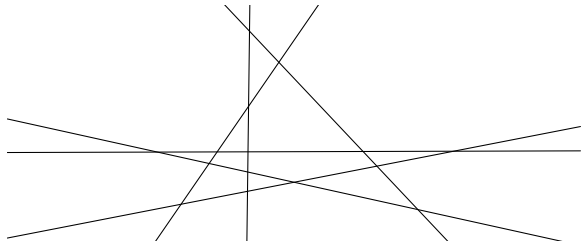
\end{center}
This configuration has only double points as singularities, and their number is the maximum possible of 15. Let $Y$ denote the 2:1 cover of $\mathbb{P}^2$ branched over the configuration $\mathcal{L}$. It is a normal surface with 15 singularities of type $\rm A_1$, namely the points sitting over the intersection of the six lines in $\mathcal{L}$. We can resolve the singularities of $Y$ by blowing up once at each singular point; this yields a smooth surface $X$, which is a K3 surface by general theory. Alternatively, we could have first blown-up $\mathbb{P}^2$ at the singular points of $\mathcal{L}$, and then taken a 2:1 cover branched over the strict transforms of the six lines (which are disjoint after performing a blow-up).

On the K3 surface $X$, we have a new configuration of curves, which we call $\mathcal{C}$, given by the union of the strict transforms of the six lines and the exceptional divisors. Notice that $\mathcal{C}$ consists of $6+15 = 21$ ($-2$)-curves which intersect at 30 points of multiplicity two, thus $n = 21$, $t_2 = 30$ and $t_r = 0$ for all $r \geq 3$. We can compute the Harbourne constant:
\[h(X; \mathcal{C}) = \frac{18-120}{30} \approx -3.4666,\]
which together with the lower bound of Theorem \ref{bound} yields
\[ -4 \leq h(X; \mathcal{C}) = -102/30 \approx -3.4666.\]
\end{example}

\begin{example}[Vinberg configuration 1]
In \cite{Vinberg}, Vinberg described the two most algebraic K3 surfaces: these are the K3 surfaces $X_4$ and $X_3$ of transcendental lattice
\[\begin{pmatrix} 2 & 0 \\ 0 & 2 \end{pmatrix} \qquad \text{and} \qquad \begin{pmatrix} 2 & 1 \\ 1 & 2 \end{pmatrix},\]
respectively. Thanks to results of Shioda and Mitani \cite{shioda-mitani}, Shioda and Inose \cite{shioda-inose}, and to the fact that the class groups of discriminants $-4$ and $-3$ are trivial, it follows that $X_4$ and $X_3$ are the unique K3 surfaces of maximum Picard number and discriminant with the minimum absolute value possible.

We start considering the surface $X_4$, and we recall how to build a model for it which is pretty convenient for our purposes. In the complex projective plane $\mathbb{P}^2$, we consider the configuration $\mathcal{L}$ of lines given by
\[ \mathcal{L}: \ xyz(x-y)(x-z)(y-z) =0.\]
This configuration has three double points and four triple points. By blowing up $\mathbb{P}^2$ in the four triple points, we obtain a del Pezzo surface $S$ with a configuration of ten ($-1$)-curves, namely the strict transforms of the six lines of $\mathcal{L}$ together with the four exceptional divisors. These ten ($-1$)-curves form a divisor $B$ on $S$, which is simple normal crossing, with only 15 double points as singularities. After blowing up these 15 double points, we get a surface $S'$, with 15 ($-1$)-curves (the exceptional divisors) and 10 ($-4$)-curves (the strict transforms of the irreducible components of $B$, which are now mutually disjoint). By taking a 2:1 cover of $S'$ branched along the 10 ($-4$)-curves, we obtain the K3 surface $X_4$, equipped with a configuration $\mathcal{V}$ of 25 smooth rational curves and 30 double points, which is described by the Petersen graph in Figure \ref{petersen}.
\begin{center}
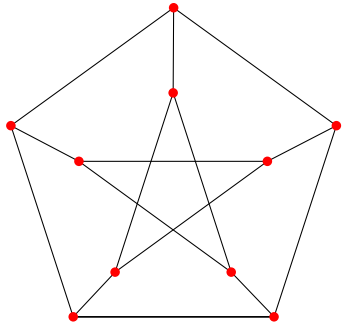
\begin{figure}[h]
\definecolor{ffqqqq}{rgb}{1,0,0}
\begin{tikzpicture}[line cap=round,line join=round,>=triangle 45,x=15,y=15,scale=1.2]
\clip(7.42,1.04) rectangle (23.01,9.13);
\draw (13.24,5.07)-- (17.2,5.07);
\draw (13.24,5.07)-- (16.44,2.74);
\draw (17.2,5.07)-- (14,2.74);
\draw (15.22,6.51)-- (14,2.74);
\draw (15.22,6.51)-- (16.44,2.74);
\draw (13.12,1.8)-- (17.34,1.8);
\draw (15.23,8.3)-- (15.22,6.51);
\draw (11.81,5.82)-- (13.24,5.07);
\draw (17.2,5.07)-- (18.65,5.82);
\draw (16.44,2.74)-- (17.34,1.8);
\draw (14,2.74)-- (13.12,1.8);
\draw (13.12,1.8)-- (17.34,1.8);
\draw (17.34,1.8)-- (18.65,5.82);
\draw (18.65,5.82)-- (15.23,8.3);
\draw (15.23,8.3)-- (11.81,5.82);
\draw (11.81,5.82)-- (13.12,1.8);
\draw (13.12,1.8)-- (17.34,1.8);
\begin{scriptsize}
\fill [color=ffqqqq] (14,2.74) circle (1.5pt);
\fill [color=ffqqqq] (16.44,2.74) circle (1.5pt);
\fill [color=ffqqqq] (17.2,5.07) circle (1.5pt);
\fill [color=ffqqqq] (15.22,6.51) circle (1.5pt);
\fill [color=ffqqqq] (13.24,5.07) circle (1.5pt);
\fill [color=ffqqqq] (13.12,1.8) circle (1.5pt);
\fill [color=ffqqqq] (17.34,1.8) circle (1.5pt);
\fill [color=ffqqqq] (18.65,5.82) circle (1.5pt);
\fill [color=ffqqqq] (15.23,8.3) circle (1.5pt);
\fill [color=ffqqqq] (11.81,5.82) circle (1.5pt);
\end{scriptsize}
\end{tikzpicture}
\caption{The Petersen graph. \label{petersen}}
\end{figure}
\end{center}
The 15 edges of the graph correspond to the exceptional divisors and the $10$ red dots correspond to curves from $B$; therefore, $n=25$, $t_2 = 30$ and $t_r = 0$ for all $r \geq 3$. We now compute the Harbourne constant for this configuration, we have:
\[h(X; \mathcal{V}) = \frac{(C_{1} + \cdots + C_{25})^2 - \sum_j m_j^2 }{s} = \frac{10-120}{30} \approx -3.666,\]
as we somehow expected from the devilish shape of the Petersen graph; together with the bound in Theorem \ref{bound}, this yields
\[-3.7333 \approx -112/30 \leq h(X;\mathcal{V}) = -110/30 \approx -3.666.\]
Vinberg's $X_{4}$ surface appears also in a different interesting context of the maximal possible cardinality of a finite  complete family of incident planes in $\mathbb{P}^{5}$ -- we refer to \cite{EPW} for details and results.
\end{example}

\begin{example}[Vinberg configuration 2]
Turning to the K3 surface $X_3$, Vinberg \cite{Vinberg} provides the reader with a particularly nice birational model, a complete intersection of a quadric and a cubic in $\mathbb{P}^4=\mathbb P^4_{(x_1:x_2:x_3:y:z)}$, which we call $Y$:
\[
\begin{cases}
y^2 = x_1^2 + x_2^2 + x_3^3 - 2(x_2 x_3 + x_1 x_3 + x_1 x_2) \\
z^3 = x_1 x_2 x_3 
\end{cases}
.\]
This model contains 9 singular points of type $\rm A_2$, namely:
\begin{align*}
&p_1=[0:1:0:1:0], \quad p_2=[0:1:0:-1:0], \quad p_3=[0:1:1:0:0],\\
&p_4=[1:1:0:0:0], \quad p_5=[0:0:1:1:0], \ \ \quad p_6=[0:0:1:-1:0], \\
&p_7=[1:0:1:0:0], \quad p_8=[1:0:0:1:0], \ \ \quad p_9=[1:0:0:-1:0]. 
\end{align*}
There are 6 lines lying in $X_3$, each of which contains three of the singular points, in such a way that each singular point is the intersection point of exactly two of the lines. More precisely, the lines are
\[L_{ijk} : \ z=x_i=y-(x_j - x_k)=0,\]
for any $i,j,k \in \lbrace 1,2,3 \rbrace$, $i \neq j\neq k \neq i$. 

The configuration consisting of these 6 lines is shown in Figure \ref{exagon}.
\begin{center}
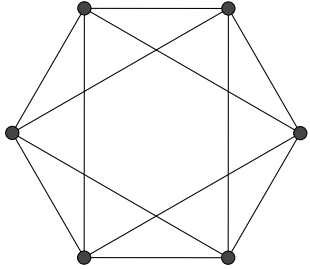
\begin{figure}[h]
\begin{tikzpicture}[line cap=round,line join=round,>=triangle 45,x=10,y=10]
\clip(-4.3,-4.82) rectangle (18.7,6.3);
\draw (8.995279811725105,5.411178196587349)-- (0.82,0.7);
\draw (11.715279811725104,0.6911781965873476)-- (3.547639905862554,5.415589098293676);
\draw (8.987639905862553,-4.024410901706327)-- (8.995279811725105,5.411178196587349);
\draw (3.54,-4.02)-- (3.547639905862554,5.415589098293676);
\draw (3.54,-4.02)-- (11.715279811725104,0.6911781965873476);
\draw (8.987639905862553,-4.024410901706327)-- (0.82,0.7);
\draw (0.82,0.7)-- (3.547639905862554,5.415589098293676);
\draw (3.547639905862554,5.415589098293676)-- (8.995279811725105,5.411178196587349);
\draw (8.995279811725105,5.411178196587349)-- (11.715279811725104,0.6911781965873476);
\draw (11.715279811725104,0.6911781965873476)-- (8.987639905862553,-4.024410901706327);
\draw (8.987639905862553,-4.024410901706327)-- (3.54,-4.02);
\draw (3.54,-4.02)-- (0.82,0.7);
\begin{scriptsize}
\draw [fill=uuuuuu] (0.82,0.7) circle (2.5pt);
\draw [fill=uuuuuu] (3.54,-4.02) circle (2.5pt);
\draw [fill=uuuuuu] (8.987639905862553,-4.024410901706327) circle (2.5pt);
\draw [fill=uuuuuu] (11.715279811725104,0.6911781965873476) circle (2.5pt);
\draw [fill=uuuuuu] (8.995279811725105,5.411178196587349) circle (2.5pt);
\draw [fill=uuuuuu] (3.547639905862554,5.415589098293676) circle (2.5pt);
\end{scriptsize}
\end{tikzpicture}
\caption{Dual graph of the six lines $L_{ij}$ on $Y$. \label{exagon}}
\end{figure}
\end{center}
We can resolve the singularities of $Y$ by blowing up twice each singular point, in order to get a smooth K3 surface, namely $X_3$: resolving each singularity yields two exceptional divisors, which are in fact ($-2$)-curves as $X_3$ is a K3 surface. The exceptional divisors together with the strict transforms of the six lines on $Y$ yields a new configuration, which we call $\mathcal{W}$. It consists of $n = 6 + 2 \cdot 9 = 24$ smooth rational curves and it has only double points as singularities, thus $t_2 = 3 \cdot 9 = 27$ and $t_r = 0$ for all $r \geq 3$. We now get the Harbourne constant:
\[h(X ;\mathcal{W}) = \frac{(C_1 + \cdots + C_{24})^2 - \sum_{r \geq 2} r^2 t_r}{s} = \frac{6 - 27 \cdot 4}{27} \approx -3.777,\]
and by Theorem \ref{bound} it follows that
\[-3.888 \approx -35/9 \leq h(X;\mathcal{W}) = -3.777.\]
\end{example}

\begin{example}[$16_6$-configuration]
Let $A$ be an abelian surface with an irreducible principal polarization. We are going to be interested in the singular Kummer surface $K$ given by the quotient of $A$ by the involution $(-1)_A$ (for a detailed account, see \cite[Chapter 10, Section 2]{birkenhake-lange}). Suppose $L$ is a symmetric line bundle on $X$ defining the principal polarization, then the map $\varphi_{L^2} : X \longrightarrow \mathbb{P}^3$ defined by the linear system $\vert L^2 \vert$ factors through an embedding of $K$ in $\mathbb{P}^3$. The singular Kummer surface $K \subset \mathbb{P}^3$ has 16 ordinary double points as singularities, namely the images of the 2-divison points. Moreover, the 16 line bundles algebraically equivalent to $L$ yield 16 planes which are tangent to $K$ and intersect $K$ along 16 conics (these planes are typically called tropes). This gives rise to the $16_6$ configuration in the Kummer surface: there are 16 points and 16 planes, each point is contained in exactly 6 planes, and each plane contains exactly 6 points. The points at which each pair of conics intersects are points of transversal intersection, as the conics lie in different planes.

Consider the blow-up at the 16 singular points of $K \subset \mathbb{P}^3$. As these are ordinary double points, one blow-up at each point is enough to resolve the singularities of $K$, and so we obtain a smooth K3 surface $X$. Since the conics of $K$ intersect transversally, locally over the blown-up points we get a tree of smooth rational ($-2$)-curves which consists of the exceptional divisor being intersected by the strict transforms of the six conics (which are now mutually disjoint).

Consider the configuration $\mathcal{C} = C_1 + \dots + C_{32} $ of ($-2$)-curves in $X$ consisting of the 16 exceptional divisors and the strict transforms of the 16 conics on $K$: these curves only meet in double points because we have blown-up all the intersection points of the conics, and the number of double points is exactly $6 \cdot 16=96$. Therefore, for the configuration $\mathcal{C}$ we have $n=32$, $t_2 = 96$ and $t_r =0$ for $r \geq 3$. The Harbourne constant is then
\[h(X; \mathcal{C}) = \frac{(C_1 + \cdots + C_{32})^2 - \sum_{r \geq 2} r^2 t_r}{s} = -8/3 \approx -2.666,\]
and together with the lower bound of Theorem \ref{bound} this shows that
\[-3.08333 \approx -296/96 \leq h(X;\mathcal{C}) = -8/3 \approx -2.666.\] 
\end{example}

\begin{example}[Schur quartic surface]
Let $S$ be the quartic surface in $\mathbb{P}^3$ given by
\[S : \ x^4 - xy^3 = z^4 - zw^3.\]
The surface $S$ is called the Schur quartic surface, and it is the surface that achieves the upper bound of 64 lines for quartic surfaces (see, for example, \cite{RS}). The 64 lines on $S$ are divided into two classes, namely lines of the 1\textsuperscript{st} kind and of the 2\textsuperscript{nd} kind. Lines of different kind can be distinguished according to the singular fibers of the fibration they induce on $S$; the singular fibers of an elliptic fibration induced by a line of the 1\textsuperscript{st} or 2\textsuperscript{nd} kind are depicted in Figures \ref{first} and \ref{second}
\begin{figure}[h]
\begin{center}
\begin{tikzpicture}[line cap=round,line join=round,>=triangle 45,x=25,y=25]
\clip(-1.2303755462612627,-1.4922814276574417) rectangle (10.508573231013045,4.0976941805684195);
\draw (-1.1,0.36)-- (10.325903810370372,0.3744775792592566);
\draw (1.9321642844444464,3.7610818666666637)-- (1.9,-1.12);
\draw (0.82,-0.92)-- (4.2,3.06);
\draw (0.29445456000000186,3.656100474074071)-- (3.48,-0.86);
\draw (7.070049075555558,3.917122391111108)-- (7.083596133333333,-1.3244204);
\draw (3.5004001073669198,2.006629230824671) node[anchor=north west] {$\times 4$};
\draw (9.287060042548877,1.9548701974151723) node[anchor=north west] {$\times 6		$};
\draw [rotate around={54.40922140811022:(7.437722543703704,1.3407835792592553)}] (7.437722543703704,1.3407835792592553) ellipse (1.9894018560916125cm and 1.4642608167523221cm);
\end{tikzpicture}
\end{center}
\caption{Singular fibers induced by a line of 1\textsuperscript{st} kind. \label{first}}
\end{figure}

\begin{figure}[h]
\begin{center}
\begin{tikzpicture}[line cap=round,line join=round,>=triangle 45,x=25,y=25]

\clip(2.8481528075140514,-0.7726442721791611) rectangle (10.051873737746615,2.657699027931594);
\draw (2.8700530084902187,0.8086963455149494)-- (10.047331709117763,0.8086963455149494);
\draw (4.713243164759446,2.2790907674008487)-- (4.713243164759446,-0.5998566096550642);
\draw (3.7419478118206824,2.2790907674008487)-- (5.6426120995857465,-0.5928688733029868);
\draw (5.68453851769821,2.2790907674008487)-- (3.7838742299331463,-0.5928688733029869);
\draw (8.179160395389857,2.3489681309216235)-- (8.179160395389857,-0.5579301915425995);
\draw (8.808056667076826,2.3629436036257783)-- (7.193889569746938,-0.5719056642467544);
\draw (7.550264123702888,2.3629436036257783)-- (9.164431221032777,-0.5719056642467543);
\draw (5.357385406669133,1.6349485254911653) node[anchor=north west] {$\times 2$};
\draw (8.546334178253565,1.6095385751199747) node[anchor=north west] {$\times 4$};
\draw (9.753306820885124,1.37449653418646) node[anchor=north west] {$\ell$};
\end{tikzpicture}
\end{center}
\caption{Singular fibers induced by a line of 2\textsuperscript{nd} kind. \label{second}}
\end{figure}
The configuration $\mathcal{S}$ of lines on $S$ counts 64 lines, 8 quadruple points, 64 triple points, and 336 double points. We can extract a subconfiguration $\check{\mathcal{S}}$ of $\mathcal{S}$, which is obtained by only considering the lines of the 2\textsuperscript{nd} kind: this configuration consists of 16 lines and has only 8 quadruple points as singularities. We can now compute the Harbourne constant in this case:
\[ h(S;  \check{\mathcal{S}}) = -8.\]
The lower bound given by Theorem \ref{bound} finally yields
\[-9 \leq h(S;\check{\mathcal{S}}) =-8.\]
It is interesting to notice that the same numerical values are achieved by means of the Bauer configuration of lines on the Fermat quartic surface
\[F : \ x^4 +y^4 + z^4 + w^4 =0,\]
as it is shown in \cite[Example 4.3]{Pokora}. Bauer configuration consists of $16$ lines and has only $8$ quadruple points as singularities, so both the Harbourne constant and the lower bound given by Theorem \ref{bound} are the same as above, namely
\[-9 \leq h(S;F) =-8.\] 
However, we remark that, in the case of the Fermat surface, all lines are of the 1\textsuperscript{st} kind.
\end{example}

\begin{example}[Double Kummer pencil]
Let $E$ and $E'$ be two elliptic curves. Recall that any elliptic curve is a 2:1 cover of $\mathbb{P}^1$ ramified at 4 points, and that the 4 ramification points are exactly the 2-torsion points of the elliptic curve (to see this, work with an elliptic curve in Legendre form). Consider the product (abelian) surface $E \times E'$, which comes with two projections onto the factors. We can see a configuration $\mathcal{E}$ of 8 elliptic curves on $E \times E'$: these are the fibers of $p$ over the 2-torsion points of $E$, which we call $C_i$ ($ 1 \leq i \leq 4$), together with the fibers of $p'$ over the 2-torsion points of $E'$, denoted by $D_j$ ($1 \leq j \leq 4$). Each $C_i$ intersects all $D_j$'s, and viceversa, thus ${\rm Sing} (\mathcal{E})$ consists of 16 points, which in turn are the 2-torsion points of $E \times E'$. We can now consider the K3 surface ${\rm Km}(E \times E')$, the Kummer surface of $E \times E'$, obtain by first quotienting by the action of $(-1)_{E \times E'}$ and then resolving the 16 singularities of type $\rm A_1$. The configuration $\mathcal{E}$  yields a configuration $\mathcal{K}$ of ($-2$)-curves on ${\rm Km}(E \times E')$, which consists of the images of the curves $C_i$ and $D_j$ ($1 \leq i,j \leq 4$) in ${\rm Km}(E \times E')$ and the 16 exceptional divisors (with their reduced structure). The configuration $\mathcal{K}$ is the \textit{double Kummer pencil configuration}, and it consists of $n=24$ ($-2$)-curves intersecting only at $t_2 = 2 \cdot 16 = 32$ double points (see Figure \ref{kummer}).

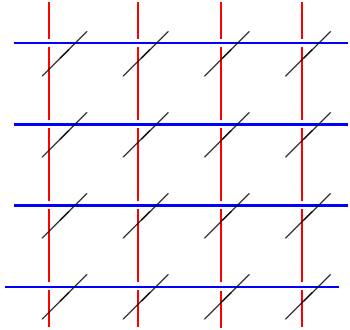
\begin{figure}[ht]
\begin{center}
\setlength{\unitlength}{1.08mm}
\begin{picture}(60, 47)
%
%

{\color{blue}
\put(7, 10){\line(1, 0){41}}}

{\color{blue}
\put(7, 20){\line(1, 0){41}}
\put(7, 30){\line(1, 0){41}}
\put(7, 40){\line(1, 0){41}}}
%
%
{\color{red}
\put(10, 5.1){\line(0, 1){4.4}}
\put(10, 10.6){\line(0, 1){8.9}}
\put(10, 20.6){\line(0, 1){8.9}}
\put(10, 30.6){\line(0, 1){8.9}}
\put(10, 40.6){\line(0, 1){4.4}}
\put(21, 5.1){\line(0, 1){4.4}}
\put(21, 10.6){\line(0, 1){8.9}}
\put(21, 20.6){\line(0, 1){8.9}}
\put(21, 30.6){\line(0, 1){8.9}}
\put(21, 40.6){\line(0, 1){4.4}}}

{\color{red}
\put(30, 5){\line(0, 1){4.4}}
\put(30, 10.6){\line(0, 1){8.9}}
\put(30, 20.6){\line(0, 1){8.9}}
\put(30, 30.6){\line(0, 1){8.9}}
\put(30, 40.6){\line(0, 1){4.4}}
\put(40, 5.1){\line(0, 1){4.4}}
\put(40, 10.6){\line(0, 1){8.9}}
\put(40, 20.6){\line(0, 1){8.9}}
\put(40, 30.6){\line(0, 1){8.9}}
\put(40, 40.6){\line(0, 1){4.4}}}
%
%

\put(8, 6){\line(1,1){5.5}}

\put(8, 16){\line(1,1){5.5}}
\put(8, 26){\line(1,1){5.5}}
\put(8, 36){\line(1,1){5.5}}
%

\put(18, 6){\line(1,1){5.5}}
\put(18, 16){\line(1,1){5.5}}
\put(18, 26){\line(1,1){5.5}}
\put(18, 36){\line(1,1){5.5}}
%

\put(28, 6){\line(1,1){5.5}}
\put(28, 16){\line(1,1){5.5}}
\put(28, 26){\line(1,1){5.5}}

\put(28, 36){\line(1,1){5.5}}
%

\put(38, 6){\line(1,1){5.5}}
\put(38, 16){\line(1,1){5.5}}
\put(38, 26){\line(1,1){5.5}}

\put(38, 36){\line(1,1){5.5}}
%
\end{picture}
\end{center}
\vskip -.3cm
\caption{The double Kummer pencil configuration. \label{kummer}}
\end{figure}

This yields the following Harbourne constant:
\[ h({\rm Km}(E \times E'), \mathcal{K}) = -14 / 4 \approx -3.5,\]
which combined with Theorem \ref{bound} results in 
\[-3.75 =-15/4 \leq h(X;\mathcal{K}) = -14/4 = -3.5.\]
\end{example}

\begin{example}[Enriques surfaces covered by symmetric quartic surfaces]
This example is borrowed from a recent paper of Mukai and Ohashi \cite{mukai-ohashi}. Let $\bar{X}$ be the quartic in $\mathbb{P}^3$ given as the zero locus of
\[ \bar{X} : \ \Big(\sum_{i < j } x_i x_j \Big)^2 = k x_0 x_1 x_2 x_3.\]
This is a singular hypersurface, with four singularities of type $\rm D_4$, namely the vertices of the fundamental tetrahedron. The coordinate planes cut $\bar{X}$ in conics with multiplicity two, which are also called \textit{tropes}, and each one of these conics passes through 3 of the singular points. After resolving the $\rm D_4$-singularities, we obtain a K3 surface $X$, which is equipped with an interesting configuration $\mathcal{C}$ of ($-2$)-curves, namely the exceptional divisors coming from the resolution of the singularities and the strict transforms of the tropes. The configuration $\mathcal{C}$ is described by the dual graph in Figure \ref{cube}, where all the intersections are points of multiplicity two, thus $n = 20$, $t_2 = 24$ and $t_r =0$ for all $r \geq 3$.
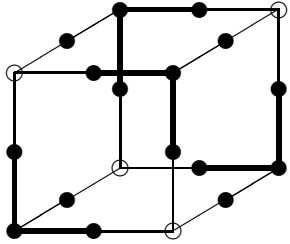
\begin{figure}[ht]
\begin{center}
\begin{picture}(120,90)
\linethickness{0.7pt}
\put(0,0){\line(1,0){60}}
\put(60,0){\line(0,1){60}}
\put(60,60){\line(-1,0){60}}
\put(0,60){\line(0,-1){60}}
\put(0,60){\line(5,3){40}}
\put(40,84){\line(1,0){60}}
\put(100,84){\line(-5,-3){40}}
\put(100,84){\line(0,-1){60}}
\put(100,24){\line(-5,-3){40}}
\put(0,0){\line(5,3){40}}
\put(40,24){\line(0,1){30}}
\put(40,24){\line(1,0){30}}
\put(0,60){\circle{6}}
\put(60,0){\circle{6}}
\put(40,24){\circle{6}}
\put(100,84){\circle{6}}
\linethickness{1.7pt}
\put(0,0){\circle*{6}}
\put(30,0){\circle*{6}}
\put(0,30){\circle*{6}}
\put(20,12){\circle*{6}}
\put(0,0){\line(0,1){30}}
\put(0,0){\line(1,0){30}}
\put(0,0){\line(5,3){20}}
\put(60,60){\circle*{6}}
\put(60,30){\circle*{6}}
\put(30,60){\circle*{6}}
\put(80,72){\circle*{6}}
\put(60,60){\line(0,-1){30}}
\put(60,60){\line(-1,0){30}}
\put(60,60){\line(5,3){20}}
\put(40,84){\circle*{6}}
\put(40,54){\circle*{6}}
\put(70,84){\circle*{6}}
\put(20,72){\circle*{6}}
\put(40,84){\line(0,-1){30}}
\put(40,84){\line(1,0){30}}
\put(40,84){\line(-5,-3){20}}
\put(100,24){\circle*{6}}
\put(100,54){\circle*{6}}
\put(70,24){\circle*{6}}
\put(80,12){\circle*{6}}
\put(100,24){\line(0,1){30}}
\put(100,24){\line(-1,0){30}}
\put(100,24){\line(-5,-3){20}}

\end{picture}
\end{center}
\caption{The configuration $\mathcal{C}$ of smooth rational curves on $X$. \label{cube}}
\end{figure}
We can compute the Harbourne constant for this configuration:
\[ h(X; \mathcal{C}) = -11/3 \approx 3.666,\]
and thanks to Theorem \ref{bound} we also see that
\[-4.333 \approx -13/3 \leq h(X; \mathcal{C}) = -11/3 \approx 3.666.\]

From the K3 surface $X$, we can construct an Enriques surface with an interesting configuration of smooth rational curves. The singular surface $\bar{X}$ is endowed with the standard Cremona transformation
\[ \varepsilon: [x_0 : x_1 : x_2 : x_3] \longmapsto [x_0^{-1} : x_1 ^{-1} : x_2 ^{-1} : x_3^{-1} ],\]
which extends to a morphism on the blown-up surface $X$. For general values of $k$ (precise conditions are given in \cite[Page 1]{mukai-ohashi}), there are no fixed points of $\varepsilon$ on $X$, and thus the quotient $X / \varepsilon =: S$ is an Enriques surface. The morphism $\varepsilon$ acts on the cube-shaped diagram in Figure \ref{cube} by point symmetry (i.e.,~symmetry with respect to the center of the cube), and thus the quotient diagram is the tetrahedron graph in Figure \ref{tetra}, also known as the 10A configuration in Mukai-Ohashi's notation. 

\begin{figure}[th]
\begin{center}
\begin{picture}(80,50)

\linethickness{0.7pt}
\put(0,0){\circle*{6}}
\put(30,0){\circle*{6}}
\put(60,0){\circle*{6}}
\put(30,50){\circle*{6}}
\put(15,25){\circle*{6}}
\put(45,25){\circle*{6}}
\put(72,24){\circle*{6}}
\put(66,12){\circle*{6}}
\put(51,38){\circle*{6}}
\put(36,12){\circle*{6}}

\put(0,0){\line(1,0){60}}
\put(0,0){\line(3,5){30}}
\put(60,0){\line(-3,5){30}}
\put(30,52){\line(3,-2){40}}
\put(72,24){\line(-1,-2){12}}
\put(0,0){\line(3,1){72}}

\end{picture}
\end{center}
\caption{The $10A$ configuration on the Enriques surface $S$. \label{tetra}}
\end{figure}
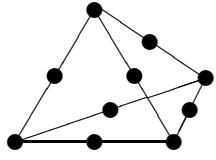

The graph describes the interaction of the images of the rational curves on $X$ modulo quotient by $\varepsilon$. We now compute the Harbourne constant for such a configuration of curves:
\[ h(X;10A) = -11/3 \approx 3.666.\]
For Enriques surfaces, the bound in Theorem \ref{bound} takes the stronger form 
\[h(X;10A) \geq -4 + \frac{2n + t_2 - 36}{s} = -13/3,\]
and thus
\[ - 4.333 \approx -13/3 \leq h(X;10A) = -11/3 \approx 3.666.\]
\end{example}

\begin{example}[A Hessian K3 surface and its Enriques quotient]
\label{ex:DvG}
The last example we would like to present uses the construction of the Hessian K3 surface associated to a cubic surface in $\mathbb{P}^3$; details can be found in \cite{DvG}. Let $S$ be a smooth cubic surface defined by the Sylvester form
\[ S : \qquad \sum_{i=0}^4 x_1^3 = \sum_{i=0}^4 x_i =0.\]
The union of the five planes in $\mathbb{P}^3$ defined by $x_i  = 0$ is called the pentahedron of $S$. The $10$ edges $L_{ijk}$ of the pentahedron are lines on $S$. We can consider the Hessian $Y$ of $S$: it is the (singular) surface defined by
\[Y :  \qquad (x_0 x_1 x_2 x_3 x_4) \sum_{i=0}^4 \frac{1}{x_i} = \sum_{i=0}^4 x_i =0.\]
The ten lines $L_{ij}$ lie on $Y$, and the vertices $P_{ijk}$ of the pentahedron are the singular points of $Y$ (double points). The desingularization $X$ of $Y$ is a K3 surface, called the Hessian K3 surface associate to $S$. There are $20$ rational curves on $X$, $10$ of which are the strict transforms of the $L_{ij}$'s, and we will call them by $N_{ij}$. The remaining ones are the curves arising from the resolution of the singularities at the $P_{ijk}$'s, and they will be denoted by $N_{ijk}$. 

We can find $20$ more rational curves on $X$ by looking at the {Eckardt points} of $S$. A smooth cubic surface has $27$ lines and $45$ plane sections which are unions of three lines. In case three coplanar lines meet at a single point, this point is called an \emph{Eckardt point}. Each Eckardt point on $S$ yields a pair of lines on $Y$ meeting at one of the $P_{ijk}$'s, for a total of $20$ lines. The strict transforms of these extra lines yield $20$ new rational curves on $X$. We can read off the intersection numbers of all these curves from \cite[Sections 1-2]{DvG}: the configuration $\mathcal{C}$ given by the $40$ aforementioned rational curves has $130$ double points. The Harbourne constant is
\[h(X;\mathcal{C}) = \frac{180 - 520}{130} = -34/13 \approx -2.615384,\]
combining with Theorem \ref{bound}, we get
\[ -2.9384615 \approx -4 +\frac{138}{130} \leq h(X;\mathcal{C}) =  -34/13 \approx -2.615384.\]

From $X$, we can cook up an Enriques surface $\bar{X}$: every Hessian quartic surface is equipped with a birational involution, which becomes a fixed-point-free morphism on the Hessian K3 surface, yielding an Enriques surface by taking the quotient \cite{Dolg-Keum}. Consider $X$ as above, equipped with its Enriques involution $\tau$: this automorphism swaps $N_{ijk}$ and $N_{lm}$ (for $\lbrace i,j,k,l,m \rbrace = \lbrace 0,1,2,3,4 \rbrace$), and it also swaps the two rational curves arising from each Eckardt point. As none of the curves in $\mathcal{C}$ is fixed by $\tau$, on the Enriques quotient $\bar{X}$, we obtain a configuration $\bar{\mathcal{C}}$ of 20 rational curves, meeting at 65 double points. As the local intersections are preserved, the Harbourne constant and its lower bound remain unchanged:
\[ -2.9384615 \approx -4 +\frac{69}{65} \leq h(\bar{X};\bar{\mathcal{C}}) =  -34/13 \approx -2.615384.\]
\end{example}

\paragraph*{\emph{Acknowledgement.}}
The present paper has grown out from discussions the authors had during the time when Piotr Pokora was visiting Roberto Laface at the Leibniz Universit\"at Hannover. It is a great pleasure to thank the Institut f\"ur Algebraische Geometrie for generous fundings and excellent working conditions. The authors would like to express their gratitude to Klaus Hulek and Matthias Sch\"utt for stimulating discussions about the topic, Igor Dolgachev for pointing out Example \ref{ex:DvG}, and Bert van Geemen for sharing his insights on the subject. The first-named author would like to thank Davide Cesare Veniani for sharing his insights on some of the examples. The second author would like to express his gratitude to  \L ukasz Sienkiewicz for pointing out Vinberg's paper \cite{Vinberg} and his construction, and to Stefan M\"uller-Stach for useful comments. At last, we would like to warmly thank the anonymous referee for his/her very useful comments and suggestions which allowed to improve our results.

The second author is partially supported by National Science Centre Poland Grant 2014/15/N/ST1/02102 and during the project he was a fellow of SFB 45 \emph{Periods, moduli spaces and arithmetic of algebraic varieties}.


\end{document}